\documentclass[12pt]{article}
\usepackage{amsmath}
\usepackage{amssymb}
\usepackage{amsthm}
\usepackage{amscd, newlfont, enumerate}
\usepackage[latin2]{inputenc}
\usepackage[dvips]{graphicx}
\usepackage[cmtip,arrow]{xy}
\usepackage[active]{srcltx}
\usepackage{color}
\usepackage{mathtools}
\usepackage{url}
\usepackage{hyperref}
\usepackage{appendix}
\usepackage[noadjust]{cite}
\usepackage{amsfonts}
\usepackage{epstopdf}

 \newtheorem{thrm}{Theorem}[section]
 
 \newtheorem{lemma}[thrm]{Lemma}
 \newtheorem{proposition}[thrm]{Proposition}
 \newtheorem{definition}[thrm]{Definition}
 \newtheorem{remark}[thrm]{Remark}

\DeclareOldFontCommand{\bf}{\normalfont\bfseries}{\mathbf}

\title{Classification of strongly positive representations of even general unitary groups.}
\author{Yeansu Kim and Ivan Mati\'{c}}

\begin{document}

\maketitle

\begin{abstract}
We explicitly construct the structure of Jacquet modules of parabolically induced representations of even unitary groups and even general unitary groups over a $p$-adic field $F$ of characteristic different than two. As an application, we obtain a classification of strongly positive discrete series representations of those groups.
\end{abstract}

{\renewcommand{\thefootnote}{} 
}

\section{Introduction}

The first purpose of this paper is to explicitly construct the Tadi\'{c}'s structure formula for the even unitary groups and the even general unitary groups. The Tadi\'{c}'s structure formula explores the Jacquet modules of parabolically induced representations. In the case of general linear groups, the Jacquet modules of parabolically induced representations are studied in \cite{ZB, Z}. The case of classical groups is of different nature due to difference of its Weyl groups and its action on the Levi subgroups. In \cite{T1}, Tadi\'c explicitly describe the structure of Jacquet modules in the cases of $Sp_{2n}, GSp_{2n},$ and $SO_{2n+1}$, later it is generalized to the cases of $SO_{2n},$ metaplectic group, and $GSpin$ groups in \cite{B, J, K1, K2, HM}. Using of the Tadi\'{c}'s structure formula, one can determine all Jacquet modules of certain classes of representations \cite{M2, MaT}.

The Tadi\'c's structure formula also happens to be extremely useful for the study of reducibility and composition series of certain induced representations which happen to be important for understanding of the unitary dual, such as standard representations and generalized principal series.

As an application of Tadi\'c's structure formula, the second purpose of this paper is to obtain a classification of strongly positive representations of even unitary groups and even general unitary groups. We note that the strongly positive representations serve as basic building blocks in the classification of discrete series of classical groups, including unitary ones, obtained in \cite{Mo, MT}, and in the classification of discrete series representatons of odd $GSpin$ groups, recently provided in \cite{KM}.

This paper is organized as follows: In Section \ref{Sec:2}, we outline standard notation. In Section \ref{Sec:3}, we obtain the Tadi\'c's structure formula for even general unitary groups, which describes the explicit structure of the Jacquet modules of the parabolically induced representations of general unitary groups. In Section \ref{Sec:4}, we obtain a classification of strongly positive representations of even general unitary groups. In the appendix, we also discuss the even unitary group case.

First author has been supported by the National Research Foundation of Korea (NRF) grant funded by the Korea government (MSIP) \\ (No.\ $2017R1C1B2010081$).

Second author has been supported by Croatian Science Foundation under the project $9364$.

\section{Notation and preliminaries}
\label{Sec:2}

\subsection{Notation}
\label{Sec:2.1}
 Let $F$ be a non-Archimedean local field of characteristic different than two and let $E/F$ be a quadratic extension of fields of characteristic different than two. Let $\Gamma=Gal(E/F)$ and we let $x \rightarrow \bar{x}$ be its non-trivial element. Choose an element $\beta \in E$ such that $E=F(\beta)$ and $\bar{\beta}=-\beta$. To define the unitary groups, we set
$$J_n = \begin{pmatrix}
 & \beta I_n \\
 -\beta I_n &
\end{pmatrix}.$$

We let ${\bf H}_n={\bf GU}(n,n)$ be the quasi-split general unitary group in $2n$ variables defined with respect to $E/F$ and $J_n$. Its $F$-points are
$${\bf H}_n(F) = \{ g \in GL_{2n}(E) | ^t\bar{g}  J_ng=\lambda J_n, \lambda \in E^{\times} \}$$
We fix $\lambda$ throughout the paper.

Let ${\bf s}$ and ${\bf M_s}$ be as in Remark \ref{rem:Levi for GU}. For a parabolic subgroup ${\bf P_s = M_s N_s}$ of ${\bf H}_n$, we denote the induced representation $\operatorname{Ind}_{\bf P_s}^{{\bf H}_n}(\rho_1 \otimes \cdots \otimes \rho_k \otimes \tau)$ by
$$ \rho_1 \times \cdots \times \rho_k \rtimes \tau$$
\noindent where each $\rho_i$ (resp. $\tau$) is a representation of some $\textbf{GL}_{n_i}(E)$ (resp. $\textbf{H}_{n-n'}(F)$). In particular, $\operatorname{Ind}_{\textbf{P}_{\textbf{s}}}^{\textbf{H}_n}$ is a functor from admissible representations of $\textbf{M}_{\textbf{s}}(F)$ to admissible representations of $\textbf{H}_n(F)$ that sends unitary representations to unitary representations. We also denote the normalized Jacquet module with respect to $\textbf{P}_{\textbf{s}}$ by $r_{\textbf{s}}(\tau)$. In particular, $r_{\textbf{s}}$ is a functor from admissible representations of $\textbf{H}_n(F)$ to admissible representations of $\textbf{M}_{\textbf{s}}(F)$.

The Grothendieck group of the category of all admissible representations of finite length of $\textbf{H}_n(F)$, i.e., a free abelian group over the set of all irreducible representations of $\textbf{H}_n(F)$ (resp. $\textbf{GL}_n(E)$) is denoted by $R_{GU}(n)$ (resp. $R_{GL}(n)$) and set $R_{GU}= \displaystyle\mathop{\oplus}\limits_{n \geq 0} R_{GU}(n), \ R_{GL}= \displaystyle\mathop{\oplus}\limits_{n \geq 0} R_{GL}(n).$

In the case of $\textbf{GL}$, we denote the induced representation $\operatorname{Ind}_{\textbf{P}'}^{\textbf{GL}_n}(\rho_1 \otimes \cdots \otimes \rho_k)$ by
$$\rho_1 \times \cdots \times \rho_k$$
\noindent such that $\textbf{P}'=\textbf{M}'\textbf{N}'$ is the standard parabolic subgroup of $\textbf{GL}_n$ where $\textbf{M}' \cong \textbf{GL}_{n_1} \times \textbf{GL}_{n_2} \times \cdots \times \textbf{GL}_{n_k}$ and each $\rho_i$ is a representation of $\textbf{GL}_{n_i}(E)$ for $i=1, \ldots, k$. We also follow the notation in \cite{ZB}. Let $\rho$ be an irreducible unitary cuspidal representation of some $\textbf{GL}_p(E)$. We define the segment, $\Delta:=[\nu^a \rho, \nu^{a+k} \rho] = \{\nu^a \rho, \nu^{a+1} \rho, \ldots \nu^{a+k} \rho\}$ where $a \in \mathbb{R}$ and $k \in \mathbb{Z}_{\geq 0}$. If $a > 0$, we call the segment $\Delta$ strongly positive.

\section{The Tadi\'c's structure formula: general unitary groups}
\label{Sec:3}

We fix the $F$-Borel subgroup {\bf B} of upper triangular matrices in ${\bf H}_n$. Then ${\bf B=TU}$, where ${\bf T}$ is a maximal torus of diagonal elements in ${\bf H}_n$ and let ${\bf A}_0$ be the maximal $F$-split subtorus of ${\bf T}$. Then,
$${\bf T}(F) = \left\lbrace \begin{pmatrix}
x_1 & & & & & \\
 & \ddots & & & & \\
 & & x_n & & & \\
 & & & \lambda \bar{x}_1^{-1}& & \\
 & & & &\ddots & \\
 & & & & & \lambda \bar{x}_n^{-1}
\end{pmatrix}
| x_i \in E^{\times}, \lambda \in F^{\times} \right\rbrace$$
and
$${\bf A}_0(F) = \left\lbrace \begin{pmatrix}
x_1 & & & & & \\
 & \ddots & & & & \\
 & & x_n & & & \\
 & & & \lambda x_1^{-1}& & \\
 & & & &\ddots & \\
 & & & & & \lambda x_n^{-1}
\end{pmatrix}
| x_i, \lambda  \in F^{\times} \right\rbrace
$$

For simplicity, we let $a(x_1, \cdots, x_n;\lambda)$ be an element of the form $diag(x_1, \cdots$, $x_n, \lambda x_1^{-1}, \cdots, \lambda x_n^{-1})$ in ${\bf A}_0(F)$.
Let $\Phi({\bf H}_n(F), {\bf A}_0(F))$ be the restricted roots of ${\bf H}_n(F)$ with respect to ${\bf A}_0(F)$ and let $\Delta:=\{\alpha_i\}_{i=1}^n$ be the set of simple roots, where $\alpha_i=e_i-e_{i+1}, 1 \leq i \leq n-1, \ \alpha_n = 2e_n-e_0$.

The Weyl group $W({\bf H}_n(F)/{\bf A}_0(F))$ is isomorphic to $S_n \rtimes \{ \pm 1 \}^n$, where $S_n$ is the permutation group of $n$ letters. More precisely, for $(ij) \in S_n$,
$$(ij) \cdot a(x_1, \cdots, x_n; \lambda) = a(x_1, \cdots, x_{i-1}, x_j, x_{i+1}, \cdots, x_{j-1}, x_i, x_{j+1}, \cdots, x_n;\lambda)$$
and for $\epsilon = (\epsilon_1, \cdots, \epsilon_n) \in \{ \pm 1 \}^n$,
$$\epsilon_i \cdot a(x_1, \cdots, x_n; \lambda) = a(x_1, \cdots, x_{i-1}, \lambda x_i^{\epsilon_i}, x_{i+1}, \cdots, x_n; \lambda).$$

\begin{remark}\label{rem:Levi for GU}
Let $\textbf{s}=(n_1, n_2, \ldots, n_k)$ be an ordered partition of some $n'$ such that $n' \leq n$ and let $\Theta = \Delta \backslash \{ \alpha_{n_1}, \alpha_{n_1+n_2}, \cdots, \alpha_{n_1 + \cdots + n_k} \}$. Let ${\bf A_s}$ be the subtorus of ${\bf A}_0$ that corresponds to $\Theta$ and let ${\bf M_s}$ be the centralizer of ${\bf A_s}$. Then, its $F$-points is of the form
$${\bf M_s}(F) = \left\lbrace \begin{pmatrix}
g_1 & & & & & & \\
 & \ddots & & & & &\\
 & & g_k & & & &\\
 & & & g & & & \\
 & & & & \lambda ^t \bar{g}_1^{-1}& & \\
 & & & & &\ddots & \\
 & & & & & & \lambda ^t \bar{g}_n^{-1}
\end{pmatrix}
| g_i \in {\bf GL}_{n_i}(E), g \in  {\bf H}_{n-n'}(F), \lambda \in F^{\times} \right\rbrace.
$$

Therefore, ${\bf M_s}(F) \cong \textbf{GL}_{n_1}(E) \times \textbf{GL}_{n_2}(E) \times \cdots \times \textbf{GL}_{n_k}(E) \times \textbf{H}_{n-n'}(F)$ and for simplicity, the element $diag(g_1, g_2, \cdots, g_k, g, \lambda ^t \bar{g}_1^{-1}, \lambda ^t \bar{g}_2^{-1}, \cdots, \lambda ^t \bar{g}_k^{-1})$ in ${\bf M_s}(F)$ is denoted by $(g_1, g_2, \cdots, g_k, g)$.
\end{remark}

Then, for an element $(g_1, g_2, \cdots, g_k, g) \in {\bf M_s}(F)$, the Weyl group $W({\bf H}_n(F)/{\bf A_s}(F))$ is a subgroup of $S_k \rtimes \{ \pm 1 \}^k$.
In particular, for $(ij) \in W({\bf H}_n(F)/{\bf A_s}(F)) \subset S_k$,
$$(ij) \cdot (g_1, g_2, \cdots, g_k, g) = (g_1, \cdots, g_{i-1}, g_j, g_{i+1}, \cdots, g_{j-1}, g_i, g_{j+1}, \cdots, g_k, g),$$

and for $\epsilon = (\epsilon_1, \cdots, \epsilon_k) \in \{ \pm 1 \}^k \subset W({\bf H}_n(F)/{\bf A_s}(F))$ with $\epsilon_i=-1, \epsilon_k=1$ for $k \neq i$,
\begin{equation}\label{eqn:Weyl group action on Levi for GU}
\epsilon \cdot (g_1, \cdots, g_i, \cdots, g_k, g) =(g_1, \cdots, \lambda ^t \bar{g}_i^{-1}, \cdots, g_k, g).
\end{equation}

Therefore, the Weyl group action on the maximal $F$-split torus ${\bf A}_0(F)$ and Levi subgroup ${\bf M_s}(F)$ of $\textbf{H}_n$ is similar to that for general symplectic groups (Note that the main difference is the Weyl group action on the Levi subgroup (\ref{eqn:Weyl group action on Levi for GU})). In \cite[Section 4]{T1}, Tadi\'c characterizes the representative element of the set $[W_{\Delta \backslash {\alpha}} \backslash W / W_{\Delta \backslash {\beta}}]$ and its explicit action on the simple roots for $\textbf{GSp}_{2n}$. We also get the same results, i.e., from Lemmas 4.1 through 4.8 of \cite{T1} in the case of even general unitary groups, since those lemmas only depend on the simple roots, Weyl group and its action on the simple roots and we also know that simple roots for $\textbf{H}_n$ is same as those for $\textbf{GSp}_{2n}$.

We now explain the Tadi\'{c}'s structure formula for $\textbf{H}_n$ and follow the notation in \cite{T1} for simplicity.
Let $i_1, i_2$ be integers which satisfy $1 \leq i_1, i_2 \leq n$. Take an integer $d$ such that $0 \leq d \leq \min \{i_1,i_2 \}$. Suppose that an integer $k$ satisfies $\max \{0, (i_1+i_2-n)-d \} \leq k \leq \min \{ i_1, i_2 \} -d.$
Let $p_n(d,k)_{i_1, i_2} \in S_n$ be defined by

\noindent \[ p_n(d,k)_{i_1, i_2}(j)  = \left\{ \begin{array}{cc}
  j & \text{for  } 1 \leq j \leq k; \\
  j+i_1-k & \text{for  } k+1 \leq j \leq i_2-d; \\
  (i_1+i_2-d+1)-j & \text{for  } i_2 - d + 1 \leq j \leq i_2; \\
  j-i_2+k & \text{for  } i_2 + 1 \leq j \leq i_1 + i_2 - d - k; \\
  j & \text{for  } i_1 + i_2 - d - k + 1 \leq j \leq n.
\end{array} \right. \]

Let $q_n(d,k)_{i_1, i_2}$ be $(p_n(d,k)_{i_1,i_2}, (\mathbf{1}_{i_2-d}, -\mathbf{1}_{d}, \mathbf{1}_{n-i_2}))$ where $\mathbf{1}_{i}=1, \ldots, 1$ (1 appears $i$ times).
Let $w=q_n(d,k)_{i_1, i_2}$. Then, for $(g_1, g_2, g_3, g_4, h) \in \textbf{GL}_k(E) \times \textbf{GL}_{i_2-d-k}(E) \times \textbf{GL}_{d}(E) \times \textbf{GL}_{i_1-d-k}(E) \times \textbf{H}_{n-i_1-i_2+d+k}(F)$, we have $w \cdot (g_1, g_2, g_3, g_4, h)= (g_1, g_4,\lambda \ ^t\bar{g}_3^{-1}, g_2, h).$

Let $\pi_i$ be an irreducible smooth representation of $\textbf{GL}_{n_i}(E)$ for $i=1,2,3,4$ and let $\sigma$ be an irreducible smooth representation of $\textbf{H}_m$. We have,
\begin{equation}\label{eqn:Weyl1 for GU}
w^{-1} \cdot (\pi_1 \otimes \pi_2 \otimes \pi_3 \otimes \pi_4 \otimes \sigma) = \pi_1 \otimes \pi_4 \otimes\ \check{\pi}_3 \otimes \pi_2 \otimes \omega_{\pi_
3} \sigma.
\end{equation}
where $\check{\pi}(g) := \pi(^t\bar{g}^{-1}).$

Set
\begin{equation}\label{eqn:Weyl2 for GU}
( \pi_1 \otimes \pi_2 \otimes \pi_3 ) \widetilde{\rtimes} (\pi_4 \otimes \sigma) =\ \check{\pi}_1 \times \pi_2 \times \pi_4 \otimes \pi_3 \rtimes \omega_{\pi_1} \sigma.
\end{equation}

Applying (\ref{eqn:Weyl1 for GU}) and (\ref{eqn:Weyl2 for GU}), we get

\begin{thrm}[Tadi\'{c}'s structure formula for general unitary groups.]\label{thm:Tadic for GU}

For $\pi \in R_{GL}(i)$ and $\sigma \in R_{GU}(n-i)$, the following structure formula holds
$$ \mu^*( \pi \rtimes \sigma ) = \mathfrak{M}^*( \pi ) \widetilde{\rtimes} \mu^*(\sigma).$$

\end{thrm}

\begin{lemma}\label{lem:Tadic for GU:example}
Let $\rho$ be an irreducible cuspidal representation of $\textbf{GL}_k(E)$ and $a, b \in \mathbb{R}$ be such that $ b-a \in \mathbb{Z}_{\geq 0}$. Let $\sigma$ be an admissible representation of finite length of $\textbf{H}_n(F)$. Write $\mu^*(\sigma)= \displaystyle\sum\limits_{\pi', \sigma'} \pi' \otimes \sigma'$.
Then $ \mathfrak{M}^*(\delta ([\nu^{a} \rho, \nu^{b} \rho])) = \displaystyle\sum\limits_{i=a-1}^b \displaystyle\sum\limits_{j=i}^b \delta([\nu^{a} \rho, \nu^{i} \rho]) \otimes \delta ([\nu^{j+1} \rho, \nu^{b} \rho]) \otimes \delta ([\nu^{i+1} \rho, \nu^{j} \rho])$ and $ \mu^*(\delta ([\nu^{a} \rho, \nu^{b} \rho]) \rtimes \sigma) = \displaystyle\sum\limits_{i=a-1}^b \displaystyle\sum\limits_{j=i}^b \displaystyle\sum\limits_{\pi', \sigma'} \delta([\nu^{-i}\ ^t \bar{\rho}^{-1}, \nu^{-a}\ ^t \bar{\rho}^{-1} ]) \times \delta ([\nu^{j+1} \rho, \nu^{b} \rho]) \times \pi' \otimes \delta ([\nu^{i+1} \rho, \nu^{j} \rho]) \rtimes \omega \sigma'.$ We omit $\delta ([\nu^{x} \rho, \nu^{y} \rho])$ if $x>y$.
\end{lemma}

We recall the definition of the strongly positive representations of $GSpin$ groups.

\begin{definition}\label{Def:sp}
An irreducible representation $\sigma$ of ${\bf H}_n(F)$ is called strongly positive if for every embedding
$$\sigma \hookrightarrow \nu^{s_1}\rho_1 \times \nu^{s_2}\rho_2 \times \cdots \times \nu^{s_k}\rho_k \rtimes \sigma_{cusp}$$
where $\rho_i, i=1, 2, \ldots, k$ are irreducible unitary cuspidal representations of $GL$, $\sigma_{cusp}$ is an irreducible cuspidal representation of ${\bf H}_{n'}(F)$ and $s_i \in \mathbb{R}, i=1, 2, \ldots, k$, then we have $s_i >0$ for each $i$.
\end{definition}

The following lemma is also useful when we explicitly calculate Jacquet modules:

\begin{lemma}
Let $\rho$ be a cuspidal representation of ${\bf GL}_k(E)$ and let $\sigma_{cusp}$ be a cuspidal representation of ${\bf H}_n(F)$. Write $\rho = \nu^{e(\rho)}\rho^u$, where $e(\rho) \in \mathbb{R}$ and $\rho^u$ is a unitary cuspidal representation. If $\rho \rtimes \sigma_{cusp}$ has a strongly positive discrete series subrepresentation, then we have
\begin{enumerate}[(i)]
\item $\widetilde{\rho^u} \cong \bar{\rho^u}$, i.e., $\rho^u$ is conjugate self-dual.
\item $\omega_{\rho}\sigma_{cusp} \cong \sigma_{cusp}$.
\end{enumerate}
\end{lemma}
\begin{proof}
Let $\sigma$ be a strongly positive subrepresentation of $\nu^{e(\rho)}
\rho^{u} \rtimes \sigma_{cusp}$. Then, $e(\rho) > 0$ since $\sigma$ is
strongly positive. If $(i)$ or $(ii)$ does not hold, then due to Lemma
2.1 in \cite{T3}, $\nu^{\alpha} \rho^{u} \rtimes \sigma_{cusp}$ is irreducible for every $\alpha$. Then we have the following embedding:
$$\sigma \hookrightarrow \nu^{e(\rho)}\rho^u \rtimes \sigma_{cusp} \cong \nu^{-e(\rho)}\widetilde{\bar{\rho^u}} \rtimes \sigma_{cusp}.$$
Since $-e(\rho) < 0$, this contradicts the strong positivity of
$\sigma$.
\end{proof}

\section{Classification of strongly positive representation of even general unitary groups}
\label{Sec:4}

In this section, we classify the strongly positive representation of even general unitary groups. We mostly follow the arguments in \cite{M1} and appendix to \cite{K1} and generalize those to our case.

\subsection{Construction of the map}
\label{Sec:4.1}

In this section, we construct the map from the set of strongly positive representations into certain induced representations. We consider the induced representations of the following form
\begin{equation}\label{eqn:ind of special type}
\delta(\Delta_1) \times \delta(\Delta_2) \times \cdots \times \delta(\Delta_k) \rtimes \sigma_{cusp}
\end{equation}
where $\Delta_1, \Delta_2, \ldots, \Delta_k$ is a sequence of strongly positive segments (See Notation \ref{Sec:2.1} for the definition of strongly positive segments) satisfying $0 < e(\Delta_1) \leq e(\Delta_2) \leq \cdots \leq e(\Delta_k)$ $($we allow $k=0$ here$)$, $\sigma_{cusp}$ an irreducible cuspidal representation of $\textbf{H}_m(F)$.

Then, we show that

\begin{thrm}\label{thm:spds embed special type}
$ $
\begin{enumerate}[(i)]
\item The induced representation $\delta(\Delta_1) \times \delta(\Delta_2) \times \cdots \times \delta(\Delta_k) \rtimes \sigma_{cusp}$ of the form (\ref{eqn:ind of special type}) has a unique irreducible subrepresentation which we denote by $\delta(\Delta_1, \ldots, \Delta_k ; \sigma_{cusp})$.
\item The strongly positive representation can be embedded into induced representation of the form (\ref{eqn:ind of special type})
\end{enumerate}
\end{thrm}

\begin{proof}
(i) and (ii) are GU analogue of Theorem 3.3 and Theorem 3.4 in \cite{M1}, respectively. Since the idea of their proofs depends on the behavior of $GL$ parts of Jacquet modules, we apply those in \cite{M1} to the case of even general unitary groups and we do not repeat here.
\end{proof}

\subsection{Classification of strongly positive representations: $D(\rho; \sigma_{cusp})$}
\label{Sec:4.2}

Let $\rho$ be a conjugate self-dual irreducible cuspidal representation of $\textbf{GL}_{n_{\rho}}(E)$ and $\sigma_{cusp}$ be an irreducible cuspidal representation of $\textbf{H}_m(F)$. Let $D(\rho;\sigma_{cusp})$ be the set of strongly positive representations whose cuspidal supports are the representation $\sigma_{cusp}$ and twists of the representation $\rho$ by positive valued characters. Let $a \geq 0$ be the unique non-negative real number such that $\nu^a \rho \rtimes \sigma_{cusp}$ reduces \cite{Si}. Furthermore, we assume that this reducibility point $a$ is in $\frac{1}{2} \mathbb{Z}$ (see (HI) of \cite{MT}, page 771). Let $k_{\rho}$ denote $\lceil a \rceil$, the smallest integer which is not smaller than $a$. In this section, we obtain the classification of strongly positive representations in $D(\rho;\sigma_{cusp})$.

In a previous section, Theorem \ref{thm:spds embed special type} implies that every strongly positive representation can be viewed as the unique irreducible subrepresentation of induced representation of the form (\ref{eqn:ind of special type}). Therefore, there exists an mapping from the set of strongly positive representations of $\textbf{H}_n(F)$ into the set of induced representations of the form (\ref{eqn:ind of special type}).

Now we further refine the image of this mapping when we restrict the mapping to $D(\rho;\sigma_{cusp})$.

\begin{thrm}\label{thm:Image of map}

Let $\sigma_{sp}$ be an irreducible strongly positive representation in $D(\rho; \sigma_{cusp})$ and consider it as the unique irreducible subrepresentation of induced representation of the form (\ref{eqn:ind of special type}). Write $\Delta_i = [\nu^{a_i} \rho, \nu^{b_i} \rho]$. Then,
$$a_i=a-k+i,\  b_1 < \ldots < b_k \ and \ k \leq \lceil a \rceil.$$

\end{thrm}

\begin{proof}
We only consider the Theorem when $a>0$. We use induction as in \cite{M1}. The cases $k=0$ and $k=1$ are exactly as in \cite{M1} and we skip the proof. We now consider the case when $k=2$. Now we have
$$\sigma_{sp} \hookrightarrow \delta([\nu^{a_1} \rho, \nu^{b_1} \rho]) \times \delta([\nu^{a_2} \rho, \nu^{b_2} \rho]) \rtimes \sigma_{cusp} $$
As in the case $k=1$, we easily show that $a_2=a$. Since $\sigma_{sp}$ is the unique irreducible subrepresentation of $\delta([\nu^{a_1} \rho, \nu^{b_1} \rho]) \times \delta([\nu^{a} \rho, \nu^{b_2} \rho]) \rtimes \sigma_{cusp}$ we also have $\sigma_{sp} \hookrightarrow \delta([\nu^{a_1} \rho, \nu^{b_1} \rho]) \rtimes \delta([\nu^{a} \rho, \nu^{b_2} \rho]; \sigma_{cusp})$. This embedding gives us the following embedding
$$\sigma_{sp} \hookrightarrow \delta([\nu^{a_1+1} \rho, \nu^{b_1} \rho]) \times \nu^{a_1}\rho \rtimes \delta([\nu^{a} \rho, \nu^{b_2} \rho]; \sigma_{cusp})$$
If $\nu^{a_1}\rho \rtimes \delta([\nu^{a} \rho, \nu^{b_2} \rho]; \sigma_{cusp})$ is irreducible, we have the embedding $\sigma_{sp} \hookrightarrow \delta([\nu^{a_1+1} \rho, \nu^{b_1} \rho]) \times \nu^{-a_1}\rho \rtimes \delta([\nu^{a} \rho, \nu^{b_2} \rho]; \sigma_{cusp})$ and this contradicts the strong positivity of $\sigma_{sp}$. Therefore, $\nu^{a_1}\rho \rtimes \delta([\nu^{a} \rho, \nu^{b_2} \rho]; \sigma_{cusp})$ is reducible.

GU analogue of Proposition 4.3 \cite{K1} implies that $a_1 \in \{a-1, a, b_2+1 \}$. Let us first consider the case when $a_1=a>1/2$. Similarly as in Proposition 3.1 in \cite{T1}, we use the following calculation of Jacquet modules:
$$r_{GL}(\nu^{a}\rho \rtimes \delta([\nu^{a} \rho, \nu^{b_2} \rho]; \sigma_{cusp}))= \nu^{-a}\widetilde{\bar{\rho}} \times \delta([\nu^{a}\rho, \nu^{b_2}\rho]) \otimes \sigma_{cusp}$$
\begin{equation}\label{thm:Jacquet module 1}
+\nu^{a}\rho \times \delta([\nu^{a}\rho, \nu^{b_2}\rho]) \otimes \sigma_{cusp}
\end{equation}
\begin{equation}\label{thm:Jacquet module 2}
\mu^*(\nu^{a}\rho \rtimes \delta([\nu^{a} \rho, \nu^{b_2} \rho]; \sigma_{cusp})) \geq \delta([\nu^{a+1}\rho, \nu^{b_2}\rho]) \otimes \nu^{a}\rho \rtimes \delta(\nu^{a}\rho, \sigma_{cusp})
\end{equation}
Due to $GU$ analogue of Lemma 4.1 in \cite{K1}, $\nu^{a}\rho \rtimes \delta(\nu^{a}\rho, \sigma_{cusp})$ is irreducible. Therefore, the irreducible subquotient of $\nu^{a}\rho \rtimes \delta([\nu^{a} \rho, \nu^{b_2} \rho]; \sigma_{cusp})$ that contains right hand side of (\ref{thm:Jacquet module 2}) in its Jacquet modules must also contain both terms in (\ref{thm:Jacquet module 1}). This implies that $\nu^{a}\rho \rtimes \delta([\nu^{a} \rho, \nu^{b_2} \rho]; \sigma_{cusp})$ is irreducible, which is a contradiction. Now we consider $a_1=a=1/2$. In this case, $GU$ analogue of Appendix of \cite{K1} (or Lemma 5.7 of \cite{K1}) implies that irreducible subrepresentation of $\delta([\nu^{1/2} \rho, \nu^{b_1} \rho]) \times \delta([\nu^{1/2} \rho, \nu^{b_2} \rho]) \rtimes \sigma_{cusp})$ is not strongly positive, which is a contradiction. Similarly as in \cite{M1}, we have a contradiction in the case $a_1=b_2+1$. The remaining case is when $a_1=a-1$, which is possible only if $a>1$. In that case, we also have $b_1 < b_2$. Completing argument of induction on $k$ is also exactly as in \cite{M1} and we skip the proof.
\end{proof}

We also show that the mapping from $D(\rho;\sigma_{cusp})$ to the set of induced representations of the form (\ref{eqn:ind of special type}) is well defined in the following theorem:

\begin{thrm}\label{thm:map from spds to ind is well defined}

Let $\sigma_{sp}$ be an irreducible strongly positive representation in $D(\rho; \sigma_{cusp})$. Then, there exist a unique set of strongly positive segments $\Delta_1, \Delta_2$, $\ldots, \Delta_k$, with $0 < e(\Delta_1) < e(\Delta_2) < \cdots < e(\Delta_k)$, and
a unique irreducible cuspidal representation $\sigma' \in R$ such that $\sigma_{sp} \simeq \delta(\Delta_1, \Delta_2, \ldots, \Delta_k;\sigma')$.
\end{thrm}

\begin{proof}

The proof is similar to \cite{M1} and we, therefore, omit the proof in this case since we constructed all the tools that we need in Section \ref{Sec:3}. 

\end{proof}

In Theorem \ref{thm:Image of map} and Theorem \ref{thm:map from spds to ind is well defined}, we construct an injective mapping from $D(\rho; \sigma_{cusp})$ into the set of induced representations of the form (\ref{eqn:ind of special type}) with refinement on the unitary exponents as in Theorem \ref{thm:Image of map}. More precisely, let $Jord_{(\rho,a)}$ stand for the set of all increasing sequences $b_1, b_2, \ldots, b_{k_{\rho}}$, where $b_i \in \mathbb{R}, b_i - a + k_{\rho} - i \in \mathbb{Z}_{\geq0}$ for $i=1, \ldots, k_{\rho}$ and $-1 < b_1 < b_2 < \cdots < b_{k_{\rho}}$. So far, we construct the following injective mapping:
$$D(\rho;\sigma_{cusp}) \hookrightarrow Jord_{(\rho,a)}$$
Now, it remains to show that this map is surjective. Let $b_1, b_2, \ldots, b_{k_{\rho}}$ denote an increasing sequence appearing in $Jord_{(\rho,a)}$. We showed in Section \ref{Sec:4.1} that the induced representation
\begin{equation}\label{SPDS:IND}
\delta([\nu^{a-k_{\rho}+1} \rho, \nu^{b_1} \rho]) \times \delta([\nu^{a-k_{\rho}+2} \rho, \nu^{b_2} \rho]) \times \cdots \times \delta([\nu^{a} \rho, \nu^{b_{k_{\rho}}} \rho]) \rtimes \sigma_{cusp}
\end{equation}
has a unique irreducible subrepresentation, which we denote by $\sigma_{(b_1, \ldots, b_{k_{\rho}} ;a)}$. \\

We apply the induction argument in \cite{M1} to show that the above subrepresentation is strongly positive and we do not repeat the argument here.

\begin{thrm}\label{thm:map is surjective}

The representation $\sigma_{(b_1, \ldots, b_{k_{\rho}};a)}$ is strongly positive.

\end{thrm}

\subsection{Classification of strongly positive representations}
\label{Sec:4.3}

Let $\rho_i$ be a conjugate self-dual irreducible cuspidal representation of $\textbf{GL}_{n_{\rho_i}}(E)$ for $i=1, \ldots, k$ and $\sigma_{cusp}$ is an irreducible cuspidal representation of $\textbf{H}_m(F)$. Let $D(\rho_1, \rho_2, \ldots, \rho_k; \sigma_{cusp})$ be the set of strongly positive representations whose cuspidal supports are the representation $\sigma_{cusp}$ and the twists of the representations $\rho_i$ by positive valued characters for $i=1, \ldots, k$. Let $a_{\rho_i} \geq 0$ be the unique non-negative real number such that $\nu^{a_{\rho_i}} \rho_i \rtimes \sigma_{cusp}$ reduces for each $i=1, \ldots, k$ \cite{Si}. Furthermore, we assume that this reducibility point $a_{\rho_i}$ is in $\frac{1}{2} \mathbb{Z}$ (see (HI) of \cite{MT}, page 771).

With Theorem \ref{thm:spds embed special type}, we use induction to prove the following two theorems as in \cite{M1}:

\begin{thrm}\label{thm:(general)spds embedding}
Let $\sigma_{sp}$ be a strongly positive representation in $D(\rho_1, \rho_2, \ldots, \rho_k$; $\sigma_{cusp})$. Then $\sigma_{sp}$ can be considered the unique irreducible subrepresentation of the following induced representation:
\begin{equation}\label{eqn:description of ind}
(\displaystyle\prod\limits_{i=1}^{k} \displaystyle\prod\limits_{j=1}^{k_i} \delta([\nu^{a_{\rho_i}- k_i +j} \rho_i , \nu^{b_j^{(i)}} \rho_i]))
\rtimes \sigma_{cusp},
\end{equation}
where $k_i \in \mathbb{Z}_{\geq0}$, $k_i \leq \lceil a_{\rho_i} \rceil, b_j^{(i)} > 0$ such that $b_j^{(i)} - a_{\rho_i} \in \mathbb{Z}_{\geq0}$, for
$i=1, \ldots, k$ $j=1, \ldots, k_i$. Also, $b_j^{(i)} < b_{j+1}^{(i)}$ for $1 \leq j \leq k_i -1$.

\end{thrm}

Theorem \ref{thm:(general)spds embedding} implies that we construct the mapping from $D(\rho_1, \rho_2, \ldots, \rho_k$; $\sigma_{cusp})$ to the set of induced representations of the form (\ref{eqn:ind of special type}). We now show that this mapping is well defined and injective.

\begin{thrm}\label{thm:(general)map is well defined}
Suppose that the representation $\sigma_{sp} \in D(\rho_1, \rho_2, \ldots, \rho_k; \sigma_{cusp})$ can be embedded as the unique irreducible subrepresentations of both representations $(\displaystyle\prod\limits_{i=1}^{k} \displaystyle\prod\limits_{j=1}^{k_i} \delta([\nu^{a_{\rho_i}- k_i +j}$ $\rho_i,$ $\nu^{b_j^{(i)}} \rho_i])) \rtimes \sigma_{cusp}$ and $(\displaystyle\prod\limits_{i=1}^{k'} \displaystyle\prod\limits_{j=1}^{k_i'} \delta([\nu^{a_{\rho_i'}- k_i' +j} \rho_i'$, $\nu^{c_{j}^{(i)}} \rho_i'])) \rtimes \sigma_{cusp}'$ as in Theorem \ref{thm:(general)spds embedding}.
Then we have $k=k', \sigma_{cusp} \cong \sigma_{cusp}'$ and $\lbrace \displaystyle\prod\limits_{j=1}^{k_i} \delta([\nu^{a_{\rho_i}- k_i +j} \rho_i , \nu^{b_j^{(i)}} \rho_i]) |$ $i = 1, \ldots, k \rbrace$ is a permutation of $\lbrace \displaystyle\prod\limits_{j=1}^{k_i'} \delta([\nu^{a_{\rho_i'}- k_i' +j} \rho_i'$, $\nu^{c_{j}^{(i)}} \rho_i']) | i=1, \ldots, k \rbrace $.
\end{thrm}
\begin{proof}
Since $\sigma_{sp} \in D(\rho_1, \rho_2, \ldots, \rho_k; \sigma_{cusp})$, $\sigma_{cusp}' \cong \sigma_{cusp}$ and $\{ \rho_i' | i=1, \ldots, k'\} \subset \{ \rho_i | i=1, \ldots, k \}$. Then, comparing the Jacquet modules, we easily see that $k=k'$ and $\lbrace \displaystyle\prod\limits_{j=1}^{k_i} \delta([\nu^{a_{\rho_i}- k_i +j} \rho_i , \nu^{b_j^{(i)}} \rho_i]) | i = 1, \ldots, k \rbrace$ is a permutation of $\lbrace \displaystyle\prod\limits_{j=1}^{k_i'} \delta([\nu^{a_{\rho_i'}- k_i' +j} \rho_i' $, $\nu^{c_{j}^{(i)}} \rho_i']) | i=1, \ldots k \rbrace$.
\end{proof}

Now we extend the above mapping to the set of all strongly positive representations of ${\bf H}(F)$. We first show the uniqueness of partial cuspidal support of strongly positive representation.
\begin{proposition}
 Let $\sigma_{sp}$ denote a strongly positive representation of ${\bf H}_n(F)$. Then there is a unique, up to isomorphism, cuspidal representation $\sigma_{cusp}$ of ${\bf H}_m(F)$ such that $\sigma_{sp}$ is a subrepresentation of $\pi \rtimes \sigma_{cusp}$, for an irreducible representation $\pi$ of $GL_{n-m}(E)$.
 \end{proposition}

\begin{proof}
 Suppose that there are non-isomorphic irreducible cuspidal representations
 $\sigma_1$ of ${\bf H}_{m_1}(F)$ and $\sigma_2$ of ${\bf H}_{m_2}(F)$, such that
 $\sigma_{sp} \hookrightarrow \pi_1 \rtimes \sigma_1$ and $\sigma_{sp}
 \hookrightarrow \pi_2 \rtimes \sigma_2$ for appropriate irreducible
 representations $\pi_1$ and $\pi_2$.

 Thus, there are cuspidal representations $\rho_1, \rho_2, \ldots, \rho_k$
 of general linear groups such that
 \begin{equation*}
 \sigma_{sp} \hookrightarrow \nu^{x_1} \rho_1 \times \nu^{x_2} \rho_2
 \times \cdots \times \nu^{x_k} \rho_k \rtimes \sigma_1.
 \end{equation*}
 Strong positivity of $\sigma_{sp}$ implies $x_i > 0$ for all $i$.

 Also, Frobenius reciprocity implies $\mu^{\ast}(\sigma_{sp}) \geq \pi_2
 \otimes \sigma_2$, which implies that
 \begin{equation*}
 \mu^{\ast}(\nu^{x_1} \rho_1 \times \nu^{x_2} \rho_2 \times \cdots \times
 \nu^{x_k} \rho_k \rtimes \sigma_1) \geq \pi_2 \otimes \sigma_2.
 \end{equation*}
 Repeated application of Lemma \ref{lem:Tadic for GU:example} implies that $\pi_2$ is an irreducible
 subquotient of $\rho'_1 \times \rho'_2 \times \cdots \times \rho'_k$,
 where $\rho'_i \in \{ \nu^{x_i} \rho_i, \nu^{-x_i}
 {}^t\overline{\rho_{i}}^{-1} \}$, for $i = 1, 2, \ldots, k$. Since
 $\sigma_1$ is not isomorphic to $\sigma_2$, using Lemma $4.2$ with obtain
 that there is an $i \in \{ 1, 2, \ldots, k \}$ such that $\rho'_i \cong
 \nu^{-x_i} {}^t\overline{\rho_{i}}^{-1}$. Since $x_i>0$, this contradicts
 strong positivity of $\sigma_{sp}$ and the proposition is proved.
\end{proof}

Furthermore, by comparing Jacquet modules as in the proof of Theorem \ref{thm:(general)map is well defined}, we also show the uniqueness of $GL$ cuspidal supports of strongly positive representation. Therefore, for any strongly positive representation $\sigma_{sp}$ of $\textbf{H}(F)$, there exists unique set of $\rho_1, \rho_2, \ldots, \rho_k$ and $\sigma_{cusp}$ such that $\sigma_{sp}$ can be considered to be the element in $D(\rho_1, \rho_2, \ldots, \rho_k; \sigma_{cusp})$.

Let $SP$ be the set of all strongly positive representations of $\textbf{H}(F)$. To see this mapping explicitly, let us collect the data from the induced representations of the form (\ref{eqn:description of ind}). Let $LJ$ be the set of $(Jord, \sigma')$ where $Jord= \displaystyle\bigcup\limits_{i=1}^{k}
\displaystyle\bigcup\limits_{j=1}^{k_i} \{ (\rho_i , b_j^{(i)}) \} $ and $\sigma'$ be an irreducible cuspidal representation in $R$ such that
\begin{enumerate}[(i)]
  \item $ \{ \rho_1, \rho_2, \ldots, \rho_k \}$ is a (possibly empty) set of mutually non-isomorphic irreducible conjugate self-dual cuspidal unitary representations of $GL$ such that $\nu^{a_{\rho_i}'} \rho_i \rtimes \sigma'$ reduces for $a_{\rho_i}' > 0$ (this defines $a_{\rho_i}'$ ),
  \item $k_i = \lceil a_{\rho_i}'\rceil$,
  \item for each $i=1, 2, \ldots, k, b_1^{(i)}, b_2^{(i)}, \ldots, b_{k_i}^{(i)}$ is a sequence of real numbers such that $a_{\rho_i}' - b_j^{(i)} \in \mathbb{Z}$, for $j=1, 2, \ldots, k_i$, and $-1 < b_1^{(i)} < b_2^{(i)} < \cdots < b_{k_i}^{(i)}$.
\end{enumerate}

Now, the last step is to show that this mapping is surjective onto $LJ$. Following \cite{M1}, we have

\begin{thrm}\label{SPDS:main}
The maps described above give a bijective correspondence between the sets SP and LJ.
\end{thrm}

\section{Appendix: Even unitary case}
\label{Appendix}

We were unable to find an appropriate reference for the structural formula for unitary groups other than \cite{MT}, where the authors, in Section 15, wrote an appropriate modification needed. Here, we shortly derive it for the purpose of obtaining a classification of strongly positive representation of even unitary groups following the same approach as in the previous sections. We emphasize that the classification of strongly positive representation of even unitary groups is also established in Sections 7 and 15 of \cite{MT}, using a different approach.

\subsection{Notation for even unitary groups}
\label{Sec:5.1}

We let ${\bf G}_n={\bf U}(n,n)$ be the quasi-split unitary group in $2n$ variables defined with respect to $E/F$ and $J_n$ and let $R_U(n)$ be the Grothendieck group of the category of all admissible representations of finite length of $\textbf{G}_n(F)$ and set $R_U= \displaystyle\mathop{\oplus}\limits_{n \geq 0} R_U(n)$. As in the even general unitary groups. Let also ${\bf B=TU}, {\bf A}_0, {\bf T}$ for ${\bf G}_n$ be defined as in Section \ref{Sec:3}.

Then,
$${\bf T}(F) = \left\lbrace \begin{pmatrix}
x_1 & & & & & \\
 & \ddots & & & & \\
 & & x_n & & & \\
 & & & \bar{x}_1^{-1}& & \\
 & & & &\ddots & \\
 & & & & & \bar{x}_n^{-1}
\end{pmatrix}
| x_i \in E^{\times} \right\rbrace$$
and
$${\bf A}_0(F) = \left\lbrace \begin{pmatrix}
x_1 & & & & & \\
 & \ddots & & & & \\
 & & x_n & & & \\
 & & & x_1^{-1}& & \\
 & & & &\ddots & \\
 & & & & & x_n^{-1}
\end{pmatrix}
| x_i \in F^{\times} \right\rbrace
$$

The $F$-points of Levi subgroups in ${\bf G}_n$ that corresponds to ${\bf s}=(n_1, n_2, \ldots$, $n_k)$ is of the form

$${\bf M_s}(F) = \left\lbrace \begin{pmatrix}
g_1 & & & & & & \\
 & \ddots & & & & &\\
 & & g_k & & & &\\
 & & & g & & & \\
 & & & & ^t \bar{g}_1^{-1}& & \\
 & & & & &\ddots & \\
 & & & & & & ^t \bar{g}_n^{-1}
\end{pmatrix}
| g_i \in {\bf GL}_{n_i}(E), g \in  {\bf G}_{n-n'}(F) \right\rbrace.
$$

\subsection{Tadi\'{c}'s structure formula for unitary groups}
\label{Sec:5.2}

Note that the Weyl group for unitary groups is isomorphic to general unitary groups. Therefore, we use the same notation for $q_n(d,k)_{i_1, i_2}$ $(=\omega)$ as in the general unitary group case. Then, for $(g_1, g_2, g_3, g_4, h) \in \textbf{GL}_k(E) \times \textbf{GL}_{i_2-d-k}(E) \times \textbf{GL}_{d}(E) \times \textbf{GL}_{i_1-d-k}(E) \times \textbf{G}_{n-i_1-i_2+d+k}(F)$, we have $w \cdot (g_1, g_2, g_3, g_4, h)= (g_1, g_4,\ ^t\bar{g}_3^{-1}, g_2, h).$

Let $\pi_i$ be an irreducible smooth representation of $\textbf{GL}_{n_i}(E)$ for $i=1,2,3,4$ and let $\sigma$ be an irreducible smooth representation of $\textbf{G}_m$. By our previous calculation,
\begin{equation}\label{eqn:Weyl1 for U}
w^{-1} \cdot (\pi_1 \otimes \pi_2 \otimes \pi_3 \otimes \pi_4 \otimes \sigma) = \pi_1 \otimes \pi_4 \otimes\ \check{\pi}_3 \otimes \pi_2 \otimes \sigma.
\end{equation}

Set
\begin{equation}\label{eqn:Weyl2 for U}
( \pi_1 \otimes \pi_2 \otimes \pi_3 ) \widetilde{\rtimes}' (\pi_4 \otimes \sigma) =\ \check{\pi}_1 \times \pi_2 \times \pi_4 \otimes \pi_3 \rtimes \sigma.
\end{equation}

We follow argument in Section \ref{Sec:3} by replacing (\ref{eqn:Weyl2 for GU}) by (\ref{eqn:Weyl2 for U}), we have

\begin{thrm}[Tadi\'{c}'s structure formula for unitary groups.]\label{thm:Tadic for U}
For $\pi \in R_{GL}(i)$ and $\sigma \in R_U(n-i)$, the following structure formula holds
$$ \mu^*( \pi \rtimes \sigma ) = \mathfrak{M}^*( \pi ) \widetilde{\rtimes}' \mu^*(\sigma).$$
\end{thrm}

\subsection{Strongly positive representation for even unitary groups}
\label{Sec:5.3}

With Tadi\'c's structure formula for unitary groups (Section \ref{Sec:5.2}), we apply the arguments as in Section \ref{Sec:4} to obtain the analogous results for even unitary groups. In this subsection, we only state the main result for even unitary groups and skip the proof since we already go through the similar arguments in Section \ref{Sec:4}.

Let $SP'$ be the set of all strongly positive representations of $\textbf{G}(F)$ and $LJ'$ be the set of $(Jord, \sigma')$ where $\sigma'$ be an irreducible cuspidal representation in $R_U$ and $Jord$ be exactly as in the case of even general unitary groups. Then, one can repeat the same arguments as before to obtain the bijective correspondence between $SP'$ and $LJ'$.

\section*{Acknowledgement}
The first author would like to thank the organizers of the workshop on Representation theory of $p$-adic groups at IISER Pune, Professors Anne-Marie Aubert, Manish Mishra, Alan Roche, Steven Spallone for their invitation and hospitality.

\begin{flushleft}
{Yeansu Kim \\
Department of Mathematics education, Chonnam National University \\
77 Yongbong-ro, Buk-gu, Gangju city, South Korea\\
E-mail: ykim@jnu.ac.kr}
\end{flushleft}

\begin{flushleft}
{Ivan Mati\'{c} \\
Department of Mathematics, University of Osijek \\ Trg Ljudevita
Gaja 6, Osijek, Croatia\\ E-mail: imatic@mathos.hr}
\end{flushleft}

\end{document}